\documentclass[12pt]{amsart}

\usepackage[mathscr]{eucal}
\usepackage{amsmath,amssymb,amscd,amsthm}%,latexsym
\usepackage{color}
\usepackage{epsfig}
\newtheorem{theorem}{Theorem}

\newtheorem{lemma}{Lemma}
\newtheorem{cor}{Corollary}

\theoremstyle{definition}
\newtheorem{remark}{Remark}

\theoremstyle{plane}

\def \beq{ \begin{equation} }
\def \eeq{\end{equation}}
\def  \q {{\bf q}}

\title{On the singularities of the curved $n$-body problem}
\begin{document}
\maketitle
\markboth{Florin Diacu}{On the singularities of the curved $n$-body problem}
\author{\begin{center}
Florin Diacu\\
\smallskip
{\small Pacific Institute for the Mathematical Sciences\\
and\\
Department of Mathematics and Statistics\\
University of Victoria\\
P.O.~Box 3060 STN CSC\\
Victoria, BC, Canada, V8W 3R4\\
diacu@math.uvic.ca\\
}\end{center}

}

\vskip0.5cm

\begin{center}
\today
\end{center}

\begin{abstract}
We study singularities of the $n$-body problem in spaces of constant curvature and generalize certain results due to Painlev\'e, Weierstrass, and Sundman. For positive curvature, some of our proofs use the correspondence between total collision solutions of the original system and their orthogonal projection---a property that offers a new method of approaching the problem in this particular case. 
\end{abstract}

\section{Introduction}

We consider the $n$-body problem in spaces of constant curvature, which we will
henceforth call the curved $n$-body problem to distinguish it from its classical  Euclidean analogue. Our goal is to study solutions that experience total collisions and, to some extent, solutions that end in some kind of hybrid singularities, i.e.~both collisional and non-collisional.

\subsection{Results}

We generalize in Section 3 a criterion proved by Paul Painlev\'e in 1897, which shows that a solution of the Euclidean $n$-body problem has a singularity if and only if the limit of the minimum distance between particles tends to zero, \cite{Pain}. Our generalization takes into account some singularities with no correspondent in the Euclidean case. But if we disregard them, a step we show to be natural from 
the physical point of view, our generalization reflects Painlev\'e's original result. This property, as well as other results obtained in our previous work \cite{Diacu1}, support the more than a century old idea that the potential given by the cotangent of the distance is the correct generalization of the Newtonian gravitational model to spaces of constant curvature.

In the second part of this paper, we establish a connection between total collision solutions and the integrals of the angular momentum in the case of positive curvature. Two classical results of the Euclidean case, the first already known to Karl Weierstrass in the 1880s, were proved by Karl Frithiof Sundman in an article\footnote{This famous 1912 paper, \cite{Sun}, was in fact an invited exposition to {\it Acta Mathematica} of two previous research articles Sundman had published in 1907, \cite{Sun1}, and 1909, \cite{Sun2}, in an obscure Finish journal called {\it Acta Societatis Scientiarum Fennicae}.} published in 1912: (i) if a solution of the $n$-body problem experiences a simultaneous total collision, all three angular momentum constants are zero, and (ii) if a triple collision occurs in the 3-body problem, the motion is planar. We will show in Sections 5 and 6 that, in a suitable setting and under some reasonable assumptions, these properties have analogues in spaces of positive constant curvature. The proofs use a lemma (developed in Section \ref{ort}), which 
establishes, under certain restrictions, the equivalence between the equations of motion and their orthogonal projection in the Euclidean space of the same dimension as the original phase space. 

\subsection{Some historical remarks}

The study of gravitation outside the Euclidean context started in the 1830s, when 
J\'anos Bolyai and Nikolai Lobachevsky extended the Newtonian 2-body problem to the hyperbolic space. These co-discoverers of hyperbolic geometry independently proposed a gravitational force proportional to the inverse area of a 2-dimensional sphere having the same radius as the distance between bodies. Ernest Schering showed in 1870 that the terms of the potential involve the hyperbolic cotangent of the hyperbolic distance, \cite{Sche}. Some years later, Wilhelm Killing adopted the cotangent potential in the positive curvature case too, \cite{Kil}. Subsequent studies of the 2-body problem proved that Kepler's laws admit natural generalizations, \cite{Lie1}, as does Bertrand's theorem, according to which there exist only two analytic central potentials for which all bounded orbits are closed, \cite{Lie2}. Other attempts (such as those of Rudolph Lipschitz, \cite{Lip}) at generalizing the problem using potentials that do not involve the cotangent of the distance failed to recover the classical properties of the Euclidean case. A more detailed history of these developments is given in \cite{Diacu1}, the first piece of work that derives and studies the equations of motion for any number of bodies, and proves Saari's conjecture\footnote{Details about Saari's conjecture can be found in \cite{Diacu2} and \cite{Diacu3}.} in the geodesic case. Paper \cite{Diacu1} is also a prerequisite for appreciating the results we prove in this article.

%%%%%%
%%%%%%         EQUATIONS  OF  MOTION
%%%%%%

\section{Equations of motion}\label{equations}

We first describe the equations of motion on $2$-dimensional manifolds of
constant curvature, namely spheres embedded in $\mathbb{R}^3$ for $\kappa>0$ 
and hyperboloids\footnote{The hyperboloid corresponds to Weierstrass's model
of hyperbolic geometry (see Appendix in \cite{Diacu1}).} embedded in the Minkovski space ${\mathbb{M}}^3$ for $\kappa<0$, and will discuss in Section \ref{3dim} the generalization to three dimensions for $k>0$. 

Consider the masses $m_1,\dots, m_n>0$ in $\mathbb{R}^3$ for $\kappa>0$
and $\mathbb{M}^3$ for $\kappa<0$, whose positions are given by the vectors
${\bf q}_i=(x_i,y_i,z_i), \ i=1,\dots, n$, and let ${\bf q}=
({\bf q}_1,\dots,{\bf q}_n)$ be the configuration of the system.
We define the gradient operator with respect to the vector ${\bf q}_i$ as
$$
\bar\nabla_{{\bf q}_i}=(\partial_{x_i},\partial_{y_i},\sigma\partial_{z_i}),
$$
where
$$
\sigma=
\begin{cases}
+1, \ \ {\rm for} \ \ \kappa>0\cr
-1, \ \ {\rm for} \ \ \kappa<0,\cr
\end{cases}
$$
and let $\bar\nabla$ denote the operator
$(\bar\nabla_{{\bf q}_1},\dots,\bar\nabla_{{\bf q}_n})$.
For the 3-dimensional vectors ${\bf a}=(a_x,a_y,a_z)$ and ${\bf b}=(b_x,b_y,b_z)$, 
we define the inner product
\begin{equation}
{\bf a}\odot{\bf b}:=(a_xb_x+a_yb_y+\sigma a_zb_z)
\end{equation}
and the cross product
\begin{equation}
{\bf a}\otimes{\bf b}:=(a_yb_z-a_zb_y, a_zb_x-a_xb_z, 
\sigma(a_xb_y-a_yb_x)).
\end{equation}

The Hamiltonian function of the system describing the motion of the $n$-body problem in spaces of constant curvature is
$$H_\kappa({\bf q},{\bf p})=T_\kappa({\bf q},{\bf p})-U_\kappa({\bf q}),$$
where 
$$
T_\kappa({\bf q},{\bf p})=\sum_{i=1}^nm_i^{-1}({\bf p}_i\odot{\bf p}_i)(\kappa{\bf q}_i\odot{\bf q}_i)
$$
is the kinetic energy and
\begin{equation}
U_\kappa({\bf q})={1\over 2}\sum_{i=1}^n\sum_{j=1,j\ne i}^n{m_im_j
(\sigma\kappa)^{1/2}{\kappa{\bf q}_i\odot{\bf q}_j}\over
[\sigma(\kappa{\bf q}_i
\odot{\bf q}_i)(\kappa{\bf q}_j\odot{\bf q}_j)-\sigma({\kappa{\bf q}_i\odot{\bf q}_j
})^2]^{1/2}}
\label{forcef}
\end{equation}
is the force function, $-U_\kappa$ representing the potential energy\footnote{In \cite{Diacu1}, we showed how this expression of $U_\kappa$ follows from the cotangent potential for $\kappa\ne 0$, and that $U_0$ is the Newtonian potential of the Euclidean problem, obtained as $\kappa\to 0$.}. Then the Hamiltonian form of the equations of motion is given by the system
\begin{equation}
\begin{cases}
\dot{\bf q}_i=
m_i^{-1}{\bf p}_i,\cr
\dot{\bf p}_i=\bar\nabla_{{\bf q}_i}U_\kappa({\bf q})-m_i^{-1}\kappa({\bf p}_i\odot{\bf p}_i)
{\bf q}_i, \ \  i=1,\dots,n, \ \kappa\ne 0,
\label{Ham}
\end{cases}
\end{equation}
the gradient of the force function having the expression
\begin{equation}
{\bar\nabla}_{{\bf q}_i}U_\kappa({\bf q})=\sum_{\substack{j=1\\ j\ne i}}^n{m_im_j(\sigma\kappa)^{3/2}(\kappa{\bf q}_j\odot{\bf q}_j)[(\kappa{\bf q}_i\odot{\bf q}_i){\bf q}_j-(\kappa{\bf q}_i\odot{\bf q}_j){\bf q}_i]\over
[\sigma(\kappa{\bf q}_i
\odot{\bf q}_i)(\kappa{\bf q}_j\odot{\bf q}_j)-\sigma({\kappa{\bf q}_i\odot{\bf q}_j
})^2]^{3/2}}.
\label{gradient}
\end{equation}

The motion of the bodies is confined to the surface of nonzero constant curvature $\kappa$, i.e.\ $({\bf q},{\bf p})\in {\bf T}^*({\bf M}_\kappa^2)^n$, where 
$$
{\bf M}^2_\kappa=\{(x,y,z)\in\mathbb{R}^3\ |\ \kappa(x^2+y^2+\sigma z^2)=1\}
$$ 
(in particular, ${\bf M}^2_1={\bf S}^2$ and ${\bf M}^2_{-1}={\bf H}^2$) and ${\bf T}^*({\bf M}_\kappa^2)^n$ is the cotangent bundle of the configuration space $({\bf M}^2_\kappa)^n$. For $\kappa>0$ we will also denote ${\bf M}^2_\kappa$ by ${\bf S}^2_\kappa$,
while for $\kappa<0$ we will denote it by ${\bf H}^2_\kappa$.

Notice that the $n$ constraints given by $\kappa{\bf q}_i\odot{\bf q}_i=1$ imply that ${\bf q}_i\odot{\bf p}_i=0$, so the $6n$-dimensional system \eqref{Ham} has  $2n$ constraints.
The Hamiltonian function provides the integral of energy,
$$
H_\kappa({\bf q},{\bf p})=h,
$$
where $h$ is the energy constant. Equations \eqref{Ham} also have the integrals 
of the angular momentum,
\begin{equation}
\sum_{i=1}^n{\bf q}_i\otimes{\bf p}_i={\bf c},
\end{equation}
where ${\bf c}=(\alpha, \beta, \gamma)$ is a constant vector. Unlike in the
Euclidean case, there are no integrals of the center of mass and linear
momentum. Their absence complicates the study of the equations because
many of the standard methods don't apply anymore. The curved
$n$-body problem is thus a fresh source for new mathematical developments.

%%%%%%
%%%%%%     SINGULARITIES
%%%%%%

\section{Singularities}

Equations \eqref{Ham} are undefined in the set ${\bf \Delta}:=\cup_{1\le i<j\le n}{\bf \Delta}_{ij}$, with
$${\bf \Delta}_{ij}:=\{{\bf q}\in({\bf M}^2_\kappa)^n\ |\ (\kappa{\bf q}_i\odot{\bf q}_j)^2=1\},$$ 
where the force function and its gradient have zero denominators. Thus $\bf\Delta$ contains all the singularities of the equations of motion. The singularity condition, $(\kappa{\bf q}_i\odot{\bf q}_j)^2=1$, suggests that we
consider two cases, so we write ${\bf \Delta}_{ij}={\bf \Delta}_{ij}^+\cup{\bf \Delta}_{ij}^-$, where
$
{\bf \Delta}_{ij}^+:=\{{\bf q}\in({\bf M}^2_\kappa)^n\ |\ \kappa{\bf q}_i\odot{\bf q}_j=1\}
$
and
$ 
{\bf \Delta}_{ij}^-:=\{{\bf q}\in({\bf M}^2_\kappa)^n\ |\ \kappa{\bf q}_i\odot{\bf q}_j=-1\}.
$
Accordingly, we define 
$$
{\bf \Delta}^+:=\cup_{1\le i<j\le n}{\bf \Delta}_{ij}^+\ \ {\rm and}\ \
{\bf \Delta}^-:=\cup_{1\le i<j\le n}{\bf \Delta}_{ij}^-.
$$
Then ${\bf \Delta}={\bf \Delta}^+\cup{\bf \Delta}^-$. The elements of ${\bf \Delta}^+$ correspond to collisions for any $\kappa\ne 0$,
whereas the elements of ${\bf \Delta}^-$ correspond to antipodal singularities for $\kappa>0$. The latter occur when two bodies are at the opposite 
ends of the same diameter of a sphere. For $\kappa<0$, antipodal 
singularities do not exist.

The set $\bf\Delta$ is related to singularities which arise from the question 
of existence and uniqueness of initial value problems. For initial conditions 
$({\bf q},{\bf p})(0)\in{\bf T}^*({\bf M}_\kappa^2)^n$ with ${\bf q}(0)\notin\bf\Delta$, standard results of the theory of differential equations ensure 
local existence and uniqueness of an analytic solution $({\bf q},{\bf p})$ defined 
on some interval $[0,t^+)$. Since the surfaces ${\bf M}^2_\kappa$ are connected, 
this solution can be analytically extended to an interval $[0,t^*)$, with $0<t^+\le t^*\le\infty$. If $t^*=\infty$, the solution is globally defined. But if $t^*<\infty$, the solution is called singular, and we say that it has a singularity at time $t^*$. 

We have seen in \cite{Diacu1} that, for $\kappa>0$, no solutions encounter antipodal singularities alone. But there are solutions that encounter collision singularities and
solutions that encounter collision-antipodal singularities, as for instance when two bodies collide at the north pole while a third body is at the south pole.

%%%%%%%  Generalization of PAINLEVE 

\subsection{Generalization of Painlev\'e's theorem} A classical result due to Paul Painlev\'e, in the Euclidean case, shows that an analytic solution defined on
$[0,t^*)$ has a singularity at $t^*$ if and only if the inferior limit of the minimum of the mutual distances vanishes when $t\to t^*$, \cite{Pain}, a detailed presentation of which can be found in \cite{Diacu0}. Under a certain assumption, we can translate this property to spaces of constant curvature. To prove these results, we start with a couple of lemmas, which generalize some properties known in the traditional literature, \cite{Win}.

%%%%  Lemma:        lim inf min     NECESSITY

\begin{lemma}
If $({\bf q},{\bf p})$ is an analytic solution of equations \eqref{Ham}, defined
on $[0,t^*)$, with $t^*$ a singularity, then
$$
\liminf_{t\to t^*}\min_{1<i\le j<n}|(\kappa{\bf q}_i\odot{\bf q}_j)^2-1|=0.
$$
\label{liminfmin}
\end{lemma}
\begin{proof}
Notice first that using the constraints $\kappa{\bf q}_i\odot{\bf q}_i=1, \ i=1,\dots,n$, the integral of energy becomes
\begin{equation}
\sum_{i=1}^nm_i^{-1}({\bf p}_i\odot{\bf p}_i)-\sum_{i=1}^n\sum_{j=1,j\ne i}^n
{m_im_j(\sigma\kappa)^{1/2}\kappa\q_i\odot\q_j\over[{\sigma-\sigma(\kappa\q_i\odot\q_j)^2]^{1/2}}}=2h.
\label{energy}
\end{equation}
Also, from the equations of motion \eqref{Ham} we can conclude that
\begin{equation}
\dot{\bf p}_i=\sum_{j=1,j\ne i}^n{m_im_j(\sigma\kappa)^{3/2}[{\bf q}_j-(\kappa{\bf q}_i\odot{\bf q}_j){\bf q}_i]\over
[\sigma-\sigma(\kappa{\bf q}_i\odot{\bf q}_j)^2]^{3/2}}-m_i^{-1}\kappa({\bf p}_i\odot{\bf p}_i){\bf q}_i.
\label{force}
\end{equation}
We can now prove the necessity of the condition in Lemma \ref{liminfmin}. For this, assume that there is a constant $c>0$ such that 
$$\liminf_{t\to t^*}\min_{1<i\le j<n}|(\kappa{\bf q}_i\odot{\bf q}_j)^2-1|\ge c.$$
Then there is a time $t_0$ in $[0,t^*)$ and constants $b_{ij}>c$, with $1\le i<j\le n$,  for which $[\sigma-\sigma(\kappa{\bf q}_i\odot{\bf q}_j)^2]^{1/2}\ge b_{ij}$ for all $t$ in $[t_0,t^*)$. Consequently equation \eqref{energy} implies that $\sum_{i=1}^nm_i^{-1}({\bf p}_i\odot{\bf p}_i)$ is bounded, so every term of this sum is bounded as well. Then equations \eqref{force} guarantee the existence of $n$ constants $\delta_i>0$ such that $|\dot{\bf p}_i|\le \delta_i, \ i=1,\dots, n$, and
therefore all $\ddot{\bf q}_i$ are bounded.

Writing the configuration vector ${\bf q}$ as a Taylor series about $t_0$ with integral
remainder,
$$
{\bf q}(t)={\bf q}(t_0)+(t-t_0)\dot{\bf q}(t_0)+\int_{t_0}^t(t-\tau)\ddot{\bf q}(\tau)d\tau,
$$
and using the fact that $\ddot{\bf q}$ is bounded, we can conclude that there is
a vector $({\bf q}^*,{\bf p}^*)$ in phase space such that
$\lim_{t\to t^*}({\bf q},{\bf p})(t)=({\bf q}^*,{\bf p}^*)$, i.e.\ the position vectors
${\bf q}_i$ and momentum vectors ${\bf p}_i$ have limiting 
positions ${\bf q}_i^*$ and ${\bf p}_i^*$, respectively. So $[\sigma-\sigma(\kappa{\bf q}_i^*\odot{\bf q}_j^*)^2]^{1/2}\ge b_{ij}$, therefore 
$|(\kappa{\bf q}_i^*\odot{\bf q}_j^*)^2-1|>0$, which means
that the distance between particles is not zero, and they are neither at a collision-antipodal
singularity if $\kappa>0$. But then the particles can keep moving. This physical conclusion suggests a contradiction with the hypothesis by showing that the solution is analytic at $t^*$. To prove this fact rigorously, notice that the domain of the solution $({\bf q},{\bf p})$ depends on the constants 
$\delta_i$, therefore on $b_{ij}$, and finally on $c$, but is independent of the
initial conditions. So by choosing the initial data, along the same solution, close enough to $t^*$ and
$({\bf q}^*,{\bf p}^*)$, the solution remains analytic at $t^*$, a conclusion
which contradicts the existence of the constant $c$ as described above, and
thus completes the proof.
\end{proof}

%%%%  Lemma:        lim inf min      SUFFICIENCY

\begin{lemma}
Assume that $({\bf q},{\bf p})$ is an analytic solution of equations \eqref{Ham}, defined on $[0,t^*)$, that is bounded away from collision-antipodal configurations if $\kappa>0$. 
Then, if
$$
\liminf_{t\to t^*}\min_{1<i\le j<n}|\kappa{\bf q}_i\odot{\bf q}_j-1|=0,
$$
$t^*$ is a singularity of the solution.
\label{liminfminrec}
\end{lemma}

\begin{proof}
Assume that $\liminf_{t\to t^*}\min_{1<i\le j<n}|\kappa{\bf q}_i\odot{\bf q}_j-1|=0$. 
Obviously, if $\ddot{\bf q}$ becomes unbounded as $t\to t^*$, then $t^*$ is a singularity,
and Lemma \ref{liminfminrec} is proved. We consequently assume $\ddot{\bf q}$ to be 
bounded. Then $\dot{\bf p}$ is bounded as $t\to t^*$, so the momentum ${\bf p}$ is bounded as well. Therefore we can conclude from equation \eqref{energy} that $U_\kappa({\bf q})$ is bounded as $t\to t^*$. 

Recall, however, that the terms defining $U_\kappa$ have denominators of the form 
$[\sigma-\sigma(\kappa{\bf q}_i\odot{\bf q}_j)^2]^{1/2}$, so when the quantities
$\sigma-\sigma(\kappa{\bf q}_i\odot{\bf q}_j)^2$ are small, the corresponding terms
of the force function become large in absolute value. But these terms have one sign if $\kappa{\bf q}_i\odot{\bf q}_j$ is near 1 and the opposite sign if it is near $-1$. As we excluded from our hypothesis solutions that come close to collision-antipodal configurations, the quantities $\kappa{\bf q}_i\odot{\bf q}_j$ are bounded away from $-1$. 
Consequently, as $\liminf_{t\to t^*}\min_{1<i\le j<n}|\kappa{\bf q}_i\odot{\bf q}_j-1|=0$, $\limsup_{t\to t^*}U_\kappa({\bf q}(t))=\infty$, a contradiction
with the conclusion drawn at the end of the previous paragraph.
So the condition assumed in Lemma \ref{liminfminrec} makes $t^*$ a singularity.
This completes the proof.
\end{proof}

%%%%%  Remark

\begin{remark}
The reason for having to exclude collision-antipodal configurations from
the hypothesis of Lemma \ref{liminfminrec} is connected to a property proved in Theorem 1 (iii) of \cite{Diacu1}. We showed there that there are choices 
of masses and initial conditions for which a 3-body problem taking place in ${\bf S}_1^2$ can have finite forces and velocities at a collision-antipodal configuration; in other words the solution remains analytic at $t^*$.  For instance, this is the case when two bodies
of mass $4m$ and a third body of mass $m$ move on a great circle of ${\bf S}_1^2$,
forming at each moment an isosceles triangle, and such that the larger
bodies collide in finite time while the smaller body reaches the diametrically opposed side of the circle. So there are orbits that do not
experience a singularity at $t^*$ but for which
$\liminf_{t\to t^*}\min_{1<i\le j<n}|\kappa{\bf q}_i\odot{\bf q}_j-1|=0$. 
\end{remark}

We can now state and prove a generalization of Painlev\'e's theorem to
spaces of constant curvature. Like for the above two lemmas, we split the result
into two statements, one proving necessity and the other sufficiency.

%%%%%%%    THEOREM:     lim min

\begin{theorem}
If $({\bf q},{\bf p})$ is an analytic solution of equations \eqref{Ham} defined on $[0,t^*)$, with $t^*$ a singularity, then
$$
\lim_{t\to t^*}\min_{1<i\le j<n}|(\kappa{\bf q}_i\odot{\bf q}_j)^2-1|=0.
$$
\end{theorem}
\begin{proof}

To prove the necessity of the condition, assume that there is a constant $c>0$ such that
$$
\limsup_{t\to t^*}\min_{1\le i<j\le n}|(\kappa{\bf q}_i\odot{\bf q}_j)^2-1|\ge c.
$$
Then there exists a sequence of times, $(t_n)_{n\in{\mathbb N}}$ (where $\mathbb N$ represents the set of positive integers), with $t_n\to t^*$
such that $|[\kappa{\bf q}_i(t_n)\odot{\bf q}_j(t_n)]^2-1|\ge c>0$ for all $i, j$ with $1\le i<j\le n$. This means that there is a positive constant $b$ for which $U({\bf q}(t_n))\le b$ for all $n\in{\mathbb N}$. Equation \eqref{energy} leads therefore to the conclusion that
$\sum_{i=1}^nm_i^{-1}({\bf p}_i(t_n)\odot{\bf p}_i(t_n))\le 2(b+h)$ for all $n\in{\mathbb N}$.
Therefore there is a constant $\alpha>0$ for which $|{\bf p}(t_n)|\le\alpha$ for all 
$n\in{\mathbb N}$. But as we already showed in the proof of Lemma \ref{liminfmin}, the
domain of the solution is independent of the choice of initial conditions. 
Consequently we can choose some initial data $t_{n_0}, {\bf q}(t_{n_0}), {\bf p}(t_{n_0})$ with $t_{n_0}$ close enough to $t^*$ to make the solution analytic at $t^*$,
which means that $t^*$ is not a singularity. This contradiction proves the
necessity of the condition, and completes the proof.
\end{proof}

%%%%%    THEOREM        Painleve   CONVERSE

\begin{theorem}
Assume that $({\bf q},{\bf p})$ is an analytic solution of equations \eqref{Ham}, defined on $[0,t^*)$, that is bounded away from collision-antipodal configurations if $\kappa>0$. 
Then, if
$$
\lim_{t\to t^*}\min_{1<i\le j<n}|\kappa{\bf q}_i\odot{\bf q}_j-1|=0,
$$
$t^*$ is a singularity of the solution.
\label{limminrec}
\end{theorem}

\begin{proof}
This result is an obvious consequence of Lemma \ref{liminfminrec}.
\end{proof}

\begin{remark}
As long as, for $\kappa>0$, solutions stay away from collision-antipodal singularities, 
the condition in Theorem 1 can be reduced to
$$\lim_{t\to t^*}\min_{1<i\le j<n}|\kappa{\bf q}_i\odot{\bf q}_j-1|=0,$$
and Theorems 1 and 2 become accurate translations of Painlev\'e's original 
result to spaces of constant curvature. For $\kappa<0$, the accurate translation
is satisfied without restrictions.
\label{away}
\end{remark}

\subsection{Remarks on the nature of singularities} In the classical 
$n$-body problem solutions can experience two kinds of singularities: collisions and pseudocollisions. The latter, which occur when the motion becomes
unbounded in finite time, have been conjectured by Painlev\'e for $n>3$, \cite{Pain}, and proved to exist for four or more bodies, \cite{Mather}, \cite{Xia}, \cite{Gerver}. We do not know whether such singularities show up in the curved $n$-body problem. The compactness of ${\bf S}_\kappa^2$ seems to exclude them for $\kappa>0$, but they may exist for $\kappa<0$. The difficulty of solving this problem is compounded by the lack of integrals of the center of mass, which played a crucial role in proving the existence of pseudocollisions in the classical case.

The collision-antipodal singularities of the curved problem raise another question. Are they due to the coordinates or the potential? The physical remarks below, point at the potential. Moreover, shifting the center of the coordinate 
system away from the center of the sphere does not remove these
singularities. But, of course, this doesn't exclude the possibility of
finding singularity-free coordinates in the future.

%%%%%%%    PHYSICAL   REMARKS

\subsection{Some physical remarks} The antipodal and the collision-anti\-podal singularities seem to obstruct the natural translation of the dynamical properties 
of the $n$-body problem from $\kappa=0$ to $\kappa> 0$. To better understand this
issue, let us first compare how the force function and its gradient vary in the Euclidean and in the curved case.

Let us start with the 2-body problem. The Euclidean force function,
$U_0({\bf q})={m_1m_2/|{\bf q}_1-{\bf q}_2|}$, is infinite at collision and
tends to zero when the distance between bodies tends to infinity. The norm
of the gradient,
$|\nabla U_0({\bf q})|$, has a similar behavior, which agrees with our perception that
the gravitational force decreases when the distance increases.

The behavior of the curved force function \eqref{forcef} and the norm of its gradient 
\eqref{gradient}, however, depend on the sign of $\kappa$. For $\kappa<0$, things 
are as in the Euclidean case. For $\kappa>0$ let's assume that one body is fixed at the north pole. Then $U_\kappa$, ranges from $+\infty$ at collision to $-\infty$ at the antipodal configuration, with $0$ when the second body is on the equator. The norm of the gradient is $+\infty$ at collision, and the smaller the farther the second body lies from collision in the northern hemisphere; it takes a positive minimum value on the equator, and is the larger the farther the second body stays from the north pole while lying in the southern hemisphere; finally, the norm of the gradient becomes $+\infty$ when the two bodies are at antipodes.

This behavior of the gradient seems to agree with our understanding of
gravitation only when the second body doesn't leave the northern hemisphere,
but not after it passes the equator. In a spherical universe 
with billions of objects ejected from a Big-Bang that took place at the north pole,  all the bodies would now still be in the northern hemisphere. But when the boundary of the expanding system approaches the equator, many bodies come close to antipodal singularities, so the potential energy becomes positive, thus
having the same sign as the kinetic energy. Then, by the integral of energy, the potential energy cannot grow beyond the value of the energy constant, which, when reached, makes the kinetic energy zero. Consequently the system stops moving. The motion then reverses from expansion to contraction, in agreement with the cosmological scenario of general relativity. So in a highly populated spherical universe, the motion is contained in the northern hemisphere, away from the equator. Therefore we can restrict the study of the case $\kappa>0$ to a hemisphere without equator and ignore solutions that reach collision-antipodal singularities.

%%%%%%%%
%%%%%%%%           ORTHOGONAL   SYSTEM
%%%%%%%%

\section{The orthogonal system}\label{ort}

The above physical remarks show that while the natural setting of the case $\kappa<0$ is the entire upper sheet of the hyperboloid, we can restrict the study of the case $\kappa>0$ to the northern hemisphere. The rest of this paper is about the positive-curvature case. 

We further introduce the equations of the orthogonal projection\footnote{In cartography, this projection is called orthographic, and was already mentioned by Hipparchus in the 2nd century B.C.}, which we will henceforth call the orthogonal system, and will show that, under certain circumstances, there is a one-to-one correspondence between its total-collision solutions and those of the original equations of motion. The orthogonal system, which will help us understand certain properties of the original equations, has the advantage of being defined in an Euclidean disk, though not endowed with the standard Euclidean distance. 

Let us also mention that the Principal Axis Theorem (see \cite{Diacu1})
allows us to use the orthogonal transformations of the sphere to keep, in some suitable basis, the original form of the equations of motion. Therefore, without loss of generality, 
we can always apply the orthogonal projection on the $xy$ plane. Similarly, we introduce 
no restrictions by assuming that the total collision we study in the next sections takes place at the north pole, $(0,0,\kappa^{-1/2})$, of ${\bf S}^2_\kappa$. 

Let $\bar{\bf q}_i=(x_i,y_i)$ be the orthogonal projection of ${\bf q}_i=(x_i,y_i,z_i)$
with the constraint $\kappa{\bf q}_i\cdot{\bf q}_i=1, \ i=1,\dots, n$, on the $xy$ plane. The momenta of the projected variables are $\bar{\bf p}_i=m_i\dot{\bar{\bf q}}_i$, and the problem is restricted to the disk $\kappa\bar{\bf q}_i\cdot\bar{\bf q}_i\le 1$. Then we obtain the orthogonal system associated to equations \eqref{Ham} by dropping the $z_i$ variables from
these equations. In other words, the orthogonal system has the form
\begin{equation}
\begin{cases}
{\dot{\bar{\bf q}}_i}=m_i^{-1}\bar{\bf p}_i,\cr
{\dot{\bar{\bf p}}_i}=\nabla_{\bar{\bf q}_i}U_\kappa(\bar{\bf q})-m_i^{-1}\kappa(
\bar{\bf p}_i\cdot\bar{\bf p}_i)\bar{\bf q}_i, \ \  i=1,\dots,n, \ \kappa > 0.
\label{Hamo}
\end{cases}
\end{equation}

The orthogonal transformation, however, introduces new (artificial) singularities in
system \eqref{Hamo}. They occur when two or more bodies reach the same diameter of the disk, a position for which at least one denominator in $\nabla_{\bar{\bf q}_i}U_\kappa(\bar{\bf q})$ vanishes. We call this a {\it diameter singularity} of system \eqref{Hamo}. In terms of \eqref{Ham}, 
diameter singularities correspond to non-singular configurations  
for which two or more bodies reach one of the geodesics that pass through the north 
pole  $(0,0,k^{-1/2})$. We call them {\it pole-geodesic configurations} of system \eqref{Ham}. Since  we are concerned with total-collision 
orbits, these artificial singularities force us to exclude from our treatment only a negligible set, as we will further see. 

Consider now the set $\mathscr C$ of solutions of system \eqref{Ham} that
encounter a total collision at time $t^*$ at the north pole, and are free of pole-geodesic
configurations in some interval $[t_0, t^*)$, where $t_0$ is solution dependent. 
Then $\mathscr C$ contains all total-collision orbits of system \eqref{Ham}, except for a lower-dimensional set. Indeed, each solution is analytic in $[t_0,t^*)$. So should the set of pole-geodesic configurations of a solution of system \eqref{Ham} have an accumulation point, then by the identity theorem of analytic functions the entire solution must be pole-geodesic, i.e.\ two or more bodies move all the time such that, at every instant, they lie on the same geodesic passing through the north pole. These constraints define the lower-dimensional set we must exclude from $\mathscr C$.

Similarly, we consider the set $\mathscr{\bar C}$ of solutions of system \eqref{Hamo} that
encounter a total collision at time $t^*$ at the origin of the disk, and are defined on
some interval $[t_0, t^*)$, with $t_0$ solution dependent. In other
words, the orbits of $\mathscr{\bar C}$ do not encounter diameter singularities
in $[t_0, t^*)$.

We can now state the following result, which proves the equivalence between the
sets $\mathscr C$ and $\mathscr{\bar C}$.

\medskip

%%%%%%%   EQUIVALENCE   LEMMA

\noindent{\bf Equivalence Lemma.}
{\it Consider the set $\mathscr C$ of solutions  
$({\bf q},{\bf p})$ of system \eqref{Ham} that encounter a total collision at time $t^*$ and 
are free of pole-geodesic configurations in some interval $[t_0,t^*)$. Also consider the set  
$\mathscr{\bar C}$ of solutions $(\bar{\bf q},\bar{\bf p})$ of the orthogonal system \eqref{Hamo} that encounter a total collision at time $t^*$ and are free of diameter singularities in the interval $[t_0,t^*)$. Then there is a one-to-one correspondence 
between $\mathscr C$ and $\mathscr{\bar  C}$ such that, for corresponding orbits, 
$(\bar{\bf q},\bar{\bf p})(t)$ is the orthogonal projection of $({\bf q},{\bf p})(t)$
for every $t$ in $[t_0,t^*)$.}

\begin{proof}
This equivalence follows from the fact that the orthogonal projection is a real analytic diffeomorphism between the hemisphere $z>0$ of ${\bf S}_\kappa^2$ and the disk of radius $\kappa^{-1/2}$. The form of the projection proves the last statement of this lemma.
Notice that the identification of a solution  $({\bf q},{\bf p})$ of system \eqref{Ham} with
a solution $(\bar{\bf q},\bar{\bf p})$ of system \eqref{Hamo} takes place in the interval
$[t_0,t^*)$, assumed to have no pole-geodesic configurations for $({\bf q},{\bf p})$
and, consequently, no diameter singularities for $(\bar{\bf q},\bar{\bf p})$.
\end{proof}

%%%%%   COROLLARY

\begin{cor}
Consider a solution $({\bf q},{\bf p})$ of the equations of motion \eqref{Ham} as in
the Equivalence Lemma
and the corresponding solution $(\bar{\bf q},\bar{\bf p})$ of the orthogonal system \eqref{Hamo}. Then if ${\bf c}=(\alpha,\beta,\gamma)$ is the constant vector of
the total angular momentum for $({\bf q},{\bf p})$, the constant vector of the total angular
momentum corresponding to $(\bar{\bf q},\bar{\bf p})$ has the form $\bar{\bf c}=(0,0,\gamma)$.
\label{ang} 
\end{cor}
\begin{proof}
The integrals of the total angular momentum for the solution $({\bf q},{\bf p})$ of
equations \eqref{Ham} are
$$
\big(\sum_{i=1}^nm_i(y_i\dot{z}_i-z_i\dot{y}_i),
\sum_{i=1}^nm_i(z_i\dot{x}_i-x_i\dot{z}_i), \sum_{i=1}^nm_i(x_i\dot{y}_i-y_i\dot{x}_i)\big)=(\alpha,\beta,\gamma).
$$
We can write the vectors $\bar{\bf q}_i$ and $\bar{\bf p}_i$ of the solution $(\bar{\bf q},\bar{\bf p})$ as
$$
\bar{\bf q}_i=(x_i,y_i,0)\ \ \ {\rm and} \ \ \  \bar{\bf p}_i=(m_i\dot{x}_i,m_i\dot{y}_i,0),
$$ 
so the integrals of the total angular momentum for the solution $(\bar{\bf q},\bar{\bf p})$ of the orthogonal system \eqref{Hamo} have the form
$$
\big(0, 0, \sigma\sum_{i=1}^nm_i(x_i\dot{y}_i-y_i\dot{x}_i)
\big)=(0,0,\gamma)=:\bar{\bf c},
$$
a fact which completes the proof.
\end{proof}

%%%%%%%%
%%%%%%%%           TOTAL   COLLISIONS
%%%%%%%%

\section{Total collisions}

In this section, we will use the Equivalence Lemma to generalize a theorem by Weierstrass and Sundman to spaces of positive constant curvature, first in two dimensions and
then in three dimensions.

%%%%%%        2-DIMENSIONAL CASE

\subsection{The 2-dimensional case}

We can now generalize a result known to Weierstrass, as communicated in a letter to 
G\"osta Mittag-Leffler in 1889 (\cite{Mit}, p.~58). Its proof was independently obtained and published by Sundman about two decades later, \cite{Sun}. The difference between this classical result and our generalization is that while in
the classical case all three components of the constant vector of the total angular momentum are zero, in the curved $n$-body problem only one component must vanish. This happens because in the latter case there are no integrals of the center of mass, so the entire system may drift in ${\bf S}^2_\kappa$ before the total collapse, thus making two of the components nonzero.

%%%%% THEOREM:  Angular momentum 0 at collision

\begin{theorem}
If $({\bf q},{\bf p})$ is a total-collision analytic solution of equations \eqref{Ham} as in the Equivalence Lemma, the constant total angular momentum vector, ${\bf c}=(\alpha,\beta,\gamma)$, has $\gamma=0$.
\label{angmom}
\end{theorem}
\begin{proof}
We will use the orthogonal projection discussed in Section \ref{ort}. Since 
the orthogonal system describes a planar motion, the constant of the
total angular momentum is a 3-vector with two zero components. The
third component is exactly the constant $\gamma$ of the original system, 
as shown in Corollary \ref{ang}. By the Equivalence Lemma it is enough to prove that, for total collision solutions of the orthogonal system, $\gamma=0$. 

Recall that $\bar{\bf q}_i:=(x_i,y_i)$ and $\bar{\bf p}_i=m_i{\dot{\bar{\bf q}}_i}, \ i=1,\dots, n$. Then the orthogonal projection of the equations of motion, the energy integral, and the integrals of the angular momentum are obtained by replacing ${\bf q}_i$ and ${\bf p}_i$ with $\bar{\bf q}_i$ and $\bar{\bf p}_i$, respectively. The last term of the inner product ${\bf a}\cdot{\bf b}=a_xb_x+a_yb_y+a_zb_z$ vanishes, so this operation becomes the standard 2-dimensional inner product ${\bf a}\cdot{\bf b}=a_xb_x+a_yb_y$, which is well defined 
in the closed disk of radius $\kappa^{-1/2}$. The gradient operator is now the standard
$\nabla$ with $\nabla_{\bar{\bf q}_i}=(\partial_{x_i},\partial_{y_i})$. Also, $U_\kappa(\bar{\bf q})$ remains a homogeneous
function of degree zero, where $\bar{\bf q}=(\bar{\bf q}_1,\dots, \bar{\bf q}_n)$.
Consequently, Euler's theorem for homogenous functions leads us to the identity
(a proof of which can be found in \cite{Diacu1}):
\begin{equation}
\bar{\bf q}_i\cdot\nabla_{\bar{\bf q}_i} U_\kappa(\bar{\bf q})=0, \ i=1,\dots, n.
\label{euler2}
\end{equation}

Consider now the moment of inertia $I({\bf q})=\sum_{i=1}^nm_i(x_i^2+y_i^2)$
of system \eqref{Ham}, as defined in \cite{Diacu1}. Notice that $I(\bar{\bf q})=I({\bf q})$, where  $I(\bar{\bf q})$
represents the moment of inertia after the orthogonal projection. Then
\begin{multline*}
\ddot{I}=2\sum_{i=1}^nm_i(\dot{x}_i^2+\dot{y}_i^2)+2\sum_{i=1}^n(x_i
\partial_{x_i}U_\kappa(\bar{\bf q})+y_i\partial_{y_i}U_\kappa(\bar{\bf q}))\\
-2\sum_{i=1}^nm_i\kappa(\dot{x}_i^2+\dot{y}_i^2)(x_i^2+y_i^2).
\end{multline*}
But from \eqref{euler2},
$x_i\partial_{x_i}U_\kappa(\bar{\bf q})+y_i\partial_{y_i}U_\kappa(\bar{\bf q})=
\bar{\bf q}_i\cdot\nabla_{\bar{\bf q}_i} U_\kappa(\bar{\bf q})=0,$
so
\begin{equation}
\ddot{I}=2\sum_{i=1}^nm_i(\dot{x}_i^2+\dot{y}_i^2)[1-\kappa(x_i^2+y_i^2)].
\label{eq}
\end{equation}
This relationship implies that $\ddot{I}>0$ for all $t$ in some interval $[t_0, t^*)$
before the total collision, where $t_0$ is sufficiently close to $t^*$, therefore $\dot{I}$ is increasing in
this interval. We can also assume that $\dot{I}$ is either positive or negative,
for if it changed sign at some point $t_1$, we could restrict our analysis to
an interval $[t_2,t^*)$ with $t_2>t_1$ in which the sign stays the same. But if $\dot{I}$ were positive, $I$
would increase as $t\to t^*$, so the limit of $I$ could not be zero at $t^*$,
and the total singularity would not take place. Therefore $\dot{I}<0$ for
all $t$ in $[t_0,t^*)$.

On the right hand side of the identity of Lagrange, 
\begin{equation}
\Big(\sum_{k=1}^ma_k^2\Big)\Big(\sum_{k=1}^mb_k^2\Big)=\Big(\sum_{k=1}^ma_kb_k\Big)^2+
\sum_{1\le k<j\le m}(a_kb_j-a_jb_k)^2,
\label{Lag}
\end{equation}
we ignore the first squared term, take $m=2n$, $a_{2k-1}=
m_k^{1/2}x_k$, $a_{2k}=m_k^{1/2}y_k$, $b_{2k-1}=m_k^{1/2}\dot{x}_k$, and
$a_{2k}=m_k^{1/2}\dot{y}_k$ for  $k=1,\dots, n$, and choose in the last sum only the terms that have the same index. Then the identity becomes the inequality
\begin{equation}
\Big[\sum_{i=1}^nm_i(x_i^2+y_i^2)\Big]\Big[\sum_{i=1}^nm_i(\dot{x}_i^2+\dot{y}_i^2)\Big]\ge\sum_{i=1}^nm_i^2(x_i\dot{y}_i-\dot{x}_iy_i)^2.
\label{fromLagr}
\end{equation}
But notice that
$$
\rho:={1\over n}\Big[\sigma\sum_{i=1}^nm_i(x_i\dot{y}_i-\dot{x}_iy_i)\Big]^2\le\sum_{i=1}^nm_i^2(x_i\dot{y}_i-\dot{x}_iy_i)^2,
$$
where, obviously, $\rho=\gamma^2/n$, therefore inequality \eqref{fromLagr} implies that
\begin{equation}
I\sum_{i=1}^nm_i(\dot{x}_i^2+\dot{y}_i^2)\ge\rho.
\label{inequality}
\end{equation}

So if we can prove that $\rho=0$, then $\gamma=\sigma \sum_{i=1}^nm_i(x_i\dot{y}_i-\dot{x}_iy_i)$ must also vanish. Assume $\rho>0$. Since,
as $t\to t^*$, a total collision takes place, $\kappa(x_i^2+y_i^2)\to 0$ 
for all $i=1,\dots,n$, so necessarily $U_\kappa(\bar{\bf q})\to\infty$. We can thus conclude from the energy relation that there exists at least an integer $i$ between 1 and $n$ for which $(\dot{x}_i^2+
\dot{y}_i^2)\to\infty$, therefore $\sum_{i=1}^n(\dot{x}_i^2+\dot{y}_i^2)\to\infty$ as $t\to t^*$. Then from identity \eqref{eq}, we have
$$
\ddot{I}\ge\sum_{i=1}^nm_i(\dot{x}_i^2+\dot{y}_i^2)
$$
in some the interval $[\tau,t^*)$, with $\tau$ sufficiently close to $t^*$. This 
inequality and \eqref{inequality} imply that in the same interval we
can write the inequality
$$ 
\ddot{I}\ge 2\rho I^{-1}.
$$
Multiplying this relationship by $-2\dot{I}>0$ and integrating the ensuing
inequality between $\tau$ and $t$, with $t$ in $(\tau,t^*)$, we obtain that
$$
4\rho\ln I(t)\ge 4\rho\ln I(\tau)+\dot{I}^2(t)-\dot{I}^2(\tau),
$$
in which, for simplicity, we used the notation $I(t):=I(\bar{\bf q}(t))$. 
This relationship implies that
$$
4\rho\ln I(t)\ge 4\rho\ln I(\tau)-\dot{I}^2(\tau).
$$
Then
$$
I(t)\ge I(\tau)e^{-\dot{I}^2(\tau)/4\rho},
$$
an inequality from which we can conclude that $I$ is bounded from below
since $\dot I(\tau)$ is finite and we assumed $\rho>0$. But then the total
collision cannot take place. Therefore $\rho$ must be zero, and the conclusion 
of the theorem follows. \end{proof}

\begin{remark}
Notice that in the proof of Theorem \ref{angmom} we used only the fact
that the potential is a homogeneous function of degree zero and that
the equations of motion possess the integrals of the angular momentum.
In fact the former condition is not necessary as long as one can
find another way to show that $\ddot I>0$ in a neighborhood of $t^*$. 
\end{remark}

\begin{remark}
In the Euclidean case, a stronger version of Theorem \ref{angmom} is true, \cite{Siegel}. The proof, however, uses the integrals of the center of mass, which don't exist in the curved $n$-body problem.

\end{remark}

%%%%%%        3-DIMENSIONAL CASE

\subsection{The 3-dimensional case}\label{3dim}

We are now focusing on the curved $n$-body problem in the hemisphere
$z>0$ of ${\bf S}_\kappa^3$.
With two exceptions, all the concepts introduced in Section \ref{equations} generalize naturally to the $3$-dimensional case. More precisely, the position vectors of the masses $m_i$ have the form ${\bf q}_i=(u_i,x_i,y_i,z_i)$, $i=1,\dots,n$, the gradient becomes
$$
\nabla_{{\bf q}_i}=(\partial_{u_i}, \partial_{x_i},\partial_{y_i}, \partial_{z_i}),
$$
and the inner product is defined as
$$
{\bf a}\cdot{\bf b}:=(a_ub_u+a_xb_x+a_yb_y+a_zb_z)
$$
for the vectors ${\bf a}=(a_u,a_x,a_y,a_z)$ and ${\bf b}=(b_u,b_x,b_y,b_z)$.
Since, in Section \ref{equations}, we expressed the kinetic energy, the force function, its gradient, the Hamiltonian, the equations of motion, and the integral of energy in terms of ${\bf q}$ and ${\bf p}$, their formal expressions stay the same. The only concept we cannot naturally extend to ${\mathbb R}^4$ is the cross product
and, consequently, the integrals of the angular momentum. To bypass this
difficulty, we will use the idea described in Section \ref{ort} of working with
the orthogonal system instead of the original equations of motion. 

So let ${\bar{\bf q}_i}=(u_i,x_i,y_i)$ be the orthogonal projection onto the hyperplane
$uxy$ of the position vector ${\bf q}_i=(u_i,x_i,y_i,z_i)$, $i=1,\dots,n$. 
The vectors ${\bf q}_i$ are defined in the solid ball of radius $\kappa^{-1/2}$ in ${\mathbb R}^3$. Equations \eqref{Hamo} then describe the motion of total-collision solutions
in terms of the orthogonal system of $n$ bodies. 
The Equivalence Lemma holds in this case too by replacing ${\bf S}^2_\kappa$
with ${\bf S}^3_\kappa$, and the 2-dimensional disk of radius $\kappa^{-1/2}$ with
the 3-dimensional ball of the same radius. 

The pole-geodesic configurations excluded from the set 
$\mathscr C$ of total collision orbits of system \eqref{Ham} occur when two 
or more bodies are on a geodesic passing through the ``north pole,'' $(0,0,0,\kappa^{-1/2})$, of ${\bf S}_\kappa^3$, where the 3-dimensional sphere is seen as a manifold embedded in ${\mathbb R}^4$. And the diameter singularities of the total-collision solutions of system
\eqref{Hamo}, excluded from $\mathscr{\bar C}$, show up when two or more bodies are on the same diameter passing through the
center of the solid ball of radius $\kappa^{-1/2}$. 

Since the orthogonal system is defined in 
the solid ball of radius $\kappa^{-1/2}$ of $\mathbb{R}^3$, we can use the standard cross product given by
$$
{\bf a}\times{\bf b}:=(a_xb_y-a_yb_x, a_yb_u-a_ub_y, 
a_ub_x-a_xb_u)
$$
for the 3-dimensional vectors ${\bf a}=(a_u,a_x,a_y)$ and ${\bf b}=(b_u,b_x,b_y)$. 
Thus, we obtain for the orthogonal system \eqref{Hamo} the three integrals of
the total angular momentum
\begin{equation}
\sum_{i=1}^n{\bar{\bf q}_i}\times{\bar{\bf p}_i}=(\bar\alpha,\bar\beta,\bar\gamma).
\label{angmomentum}
\end{equation}
To prove equations \eqref{angmomentum}, notice first that
$$
\sum_{i=1}^n\nabla_{\bar{\bf q}_i}U_\kappa(\bar{\bf q})\times{\bar{\bf q}}_i
=\sum_{i=1}^n\sum_{j=1,j\ne i}^nA_{ij}{\bar{\bf q}}_i\times{\bar{\bf q}}_j
+ \sum_{i=1}^n B_{ij} {\bar{\bf q}}_i\times{\bar{\bf q}}_i={\bf 0},
$$
where $A_{ij}$ and $B_{ij}$ are symmetric in $i$ and $j$, 
the last equality following from the skew-symmetry of the cross product. Then, if $\times$-multiplying $m_i\ddot{\bar{\bf q}}_i$ from \eqref{Hamo} by ${\bar{\bf q}}_i$ and adding for all $i$ from $1$ to $n$, we obtain
\begin{multline*}
\sum_{i=1}^nm_i{\ddot{\bar{\bf q}}_i}\times{\bar{\bf q}}_i=\sum_{i=1}^n\nabla_{\bar{\bf q}_i}U_\kappa(\bar{\bf q})\times{\bar{\bf q}}_i-\sum_{i=1}^nm_i^{-1}\kappa(\bar{\bf p}_i\cdot\bar{\bf p}_i)\bar{\bf q}_i\times{\bar{\bf q}}_i={\bf 0}.
\end{multline*}
Integrals \eqref{angmomentum} follow by integrating the identity $\sum_{i=1}^nm_i{\ddot{\bar{\bf q}}_i}\times{\bar{\bf q}}_i={\bf 0}$.

We can now generalize Theorem \ref{angmom} to three dimensions. For the same reasons mentioned in Section \ref{ort}, we can assume without loss of generality that the total collision takes place in ${\bf S}^3_\kappa$ at $(0,0,0,\kappa^{-1/2}$), the ``north pole'' of the 3-sphere. 
We will show that for a total collision solution of the original system in ${\bf S}^3_\kappa$, all three constants of the total angular momentum of the orthogonal system must vanish.

%%%%% THEOREM:    3-DIM  ANGULAR  MOMENTUM

\begin{theorem}
Let $({\bf q},{\bf p})$ be a total-collision analytic solution of equations \eqref{Ham} as in the Equivalence Lemma, and let  $(\bar{\bf q},\bar{\bf p})$ be the corresponding solution of the orthogonal system \eqref{Hamo}. Then all three components $\bar\alpha, \bar\beta$, and $\bar\gamma$ of the constant total angular momentum vector belonging to the solution $(\bar{\bf q},\bar{\bf p})$  are zero.
\label{angmom3}
\end{theorem}
\begin{proof}
The idea of the proof is the same as for Theorem \ref{angmom}.
Let us start by noticing that, according to the 3-dimensional version of the
Equivalence Lemma, if $({\bf q},{\bf p})$ has no collision singularities or pole-geodesic
configurations in the interval $[t_0,t^*)$, but experiences a total collision
at the point $(0,0,0,\kappa^{-1/2})$ of ${\bf S}^3_\kappa$ at time $t^*$, 
then $(\bar{\bf q},\bar{\bf p})$ ends in a total collision at time $t^*$ at the origin of ${\mathbb R}^3$, and vice versa. We can thus define the moment of inertia $\bar I=\sum_{i=1}^n m_i(u_i^2+x_i^2+y_i^2)$ and use the same steps as in the proof of
Theorem \ref{angmom} to prove that
$$
\ddot{\bar I}=2\sum_{i=1}^nm_i(\dot{u}_i^2+\dot{x}_i^2+\dot{y}_i^2)[1-\kappa
(u_i^2+x_i^2+y_i^2)],
$$
in order to have the inequality $\dot{\bar I}<0$ in some interval $[t_0,t^*)$, with $t_0>0$. In the identity of Lagrange, \eqref{Lag}, we ignore the first squared sum of the right
hand side, take $m=3n$, $a_{3k-2}=m_k^{1/2}u_k$,
$a_{3k-1}=m_k^{1/2}x_k$, $a_{3k}=m_k^{1/2}y_k$,
$b_{3k-2}=m_k^{1/2}\dot{u}_k$, $b_{3k-1}=m_k^{1/2}\dot{x}_k$, and
$b_{3k}=m_k^{1/2}\dot{y}_k$ for $k=1,\dots,n$, and pick from the second 
sum of the right hand side only the terms that have the same index. Then the identity turns into the inequality
\begin{multline}
\Big[\sum_{i=1}^nm_i(u_i^2+x_i^2+y_i^2)\Big]\Big[\sum_{i=1}^nm_i(
\dot{u}_i^2+\dot{x}_i^2+\dot{y}_i^2)\Big]\ge\\
\sum_{i=1}^nm_i^2[(x_i\dot{y}_i-\dot{x}_iy_i)^2+(y_i\dot{u}_i-\dot{y}_iu_i)^2+
(u_i\dot{x}_i-\dot{u}_ix_i)^2].
\label{fromLag}
\end{multline}
But notice that
\begin{multline*}
\bar\rho:={1\over n}\Big[\sum_{i=1}^nm_i(x_i\dot{y}_i-\dot{x}_iy_i)\Big]^2
+{1\over n}\Big[\sum_{i=1}^nm_i(y_i\dot{u}_i-\dot{y}_iu_i)\Big]^2+\\
{1\over n}\Big[\sum_{i=1}^nm_i(u_i\dot{x}_i-\dot{u}_ix_i)\Big]^2\le\\
\sum_{i=1}^nm_i^2[(x_i\dot{y}_i-\dot{x}_iy_i)^2+(y_i\dot{u}_i-\dot{y}_iu_i)^2+
(u_i\dot{x}_i-\dot{u}_ix_i)^2],
\end{multline*}
where, obviously, $\bar\rho=(\bar\alpha^2+\bar\beta^2+\bar\gamma^2)/n$.
Thus inequality \eqref{fromLag} implies that 
\begin{equation*}
\bar I\sum_{i=1}^nm_i(\dot{u}_i^2+\dot{x}_i^2+\dot{y}_i^2)\ge\bar\rho.
\label{ineq}
\end{equation*}
To prove that $\bar\rho=0$, we need to follow the same steps taken from this
point on in the proof of Theorem \ref{angmom}. We are thus led to the
conclusion that $\bar\alpha=\bar\beta=\bar\gamma=0$.
\end{proof}

We further show how Theorem \ref{angmom3} can be used to draw some conclusions about the behavior of certain solutions of the original system.

%%%%%%
%%%%%%         TRIPLE  COLLISIONS
%%%%%%

\section{Triple collisions}

In his 1912 paper, Sundman also proved that if a solution of the Euclidean 3-body
problem ends in a triple collision, the motion must be planar. Again, his proof rests 
on the integrals of the center of mass.

In the curved 3-body problem we might encounter solutions in which the triangle having the bodies at its vertices moves in ${\bf S}^3_\kappa$ without
remaining confined to a 2-dimensional hemisphere. But we will show that if the collision point is fixed in a sense we will make precise, then the motion must indeed take place on a 2-dimensional hemisphere.

%%%%%     THEOREM  ON   TRIPLE  COLLISIONS

\begin{theorem}
Let $({\bf q},{\bf p})$ be a total-collision analytic solution of equations \eqref{Ham} as in the Equivalence Lemma, with $n=3$, and such that the coordinates of the corresponding solution $(\bar{\bf q},\bar{\bf p})$ of the orthogonal system \eqref{Hamo} satisfy the conditions
\begin{equation}
\sum_{i=1}^3m_i{u}_i(t)=\sum_{i=1}^3m_i{x}_i(t)=\sum_{i=1}^3m_i{y}_i(t)=0
\label{mass}
\end{equation}
for all $t$ in $[t_0,t^*)$. Then the configuration ${\bf q}$ is confined to a 2-dimension\-al hemisphere of curvature $\kappa>0$.
\label{manif}
\end{theorem}
\begin{proof}
Let $({\bf q},{\bf p})$ be a solution as above, i.e.~ one for
which the coordinates of the vectors ${\bf q}_i=(u_i,x_i,y_i,z_i)$ satisfy the 
constraints $u_i^2+x_i^2+y_i^2+z_i^2=\kappa^{-1}$, $z_i>0$, and let
$(\bar{\bf q},\bar{\bf p})$ be the corresponding solution of the orthogonal system,
verifying conditions \eqref{mass}. For system \eqref{Hamo} we can now follow Sundman's idea of proof, although the form of this system is different from the 
one of the Euclidean 3-body problem. Using \eqref{mass}, we can assume that at time $t=t_0$ the three bodies of the solution $(\bar{\bf q},\bar{\bf p})$ lie in the plane $y=0$ of $\mathbb{R}^3$. Since, by Theorem \ref{angmom3}, the constant vector of
the angular momentum is such that $\bar{\alpha}=\bar{\beta}=\bar{\gamma}=0$,
we can write that $\sum_{i=1}^3m_i{u}_i(t_0)\dot{{y}}_i(t_0)=\sum_{i=1}^3m_i{x}_i(t_0)\dot{{y}}_i(t_0)=0,$
and from the last equation of \eqref{mass}, we can also conclude that
$\sum_{i=1}^3m_i\dot{{y}}_i(t_0)=0$. The algebraic system of these three 
equations of unknowns $\dot{{y}}_1(t_0), \dot{{y}}_2(t_0)$, and $\dot{{y}}_3(t_0)$  
leads us to two possibilities: either $\dot{{y}}_1(t_0)=\dot{{y}}_2(t_0)=\dot{{y}}_3(t_0)=0$ or 
$$
\det\begin{bmatrix}
{u}_1(t_0) & {u}_2(t_0) & {u}_3(t_0) \\
{x}_1(t_0) & {x}_2(t_0) & {x}_3(t_0) \\
1 & 1 & 1
\end{bmatrix}=0.
$$
In the former case, the solution $(\bar{\bf q},\bar{\bf p})$ is confined to 
the plane $y=0$. In the latter case, the bodies are initially on a straight line
passing through the origin of the coordinate system. But we
excluded the latter situation by eliminating diameter singularities. Thus $(\bar{\bf q},\bar{\bf p})$ is a planar solution, therefore the position vectors of the
original system are of the form ${\bf q}_i=(u_i,x_i, y_i, z_i)$, with $y_i=0, z_i>0$, and $u_i^2+x_i^2+z_i^2=\kappa^{-1}, \ i=1,2,3$, so the bodies move on a 2-dimensional hemisphere.
\end{proof}

\begin{remark}
Conditions \eqref{mass} are naturally satisfied by some classes of solutions
of the curved $n$-body problem, such as the elliptic relative equilibria, but
not satisfied by others, the hyperbolic relative equilibria among them. The
existence of these orbits was proved in  \cite{Diacu1}.
\end{remark}

\begin{remark}
In the proof of Theorem \ref{manif} we used only the homogeneity of the potential, conditions \eqref{mass}, and the integrals of the total angular momentum. So this result is valid for more general potentials than the ones considered in this paper.
\end{remark}

%%%%%%
%%%%%%          BIBLIOGRAPHY
%%%%%%

\bigskip

%\textcolor{red}{{\bf Reminders}\\
%{\bf 1.}
%}


\begin{thebibliography}{99}

\bibitem{Diacu0} F.~Diacu, Singularities of the $N$-body problem, in {\it Classical and Celestial Mechanics: The Recife Lectures}, H.~Cabral and F.~Diacu, eds., Princeton University Press, Princeton, N.J., 2002, pp.~35-62.

\bibitem{Diacu2} F.~Diacu, E.~P\'erez-Chavela, and M.~Santoprete, Saari's conjecture for the collinear $n$-body problem, {\it Trans.~Amer.~Math.~Soc.} {\bf 357}, 10 (2005), 4215-4223. 

\bibitem{Diacu3} F.~Diacu, T.~Fujiwara, E.~P\'erez-Chavela, and M.~Santoprete,
Saari's homographic conjecture of the 3-body problem, {\it Trans.~Amer.~Math.~Soc.} {\bf 360}, 12 (2008), 6447-6473.

\bibitem{Diacu1} F.~Diacu, E.~P\'erez-Chavela, and M.~Santoprete,
The $n$-body problem in spaces of constant curvature, arXiv:0807.1747
(2008), 54 p.

\bibitem{Gerver} J.~Gerver, The existence of pseudocollisions in the plane,
{\it J.\ Differential Equations} {\bf 89} (1991), 1-68.

\bibitem{Kil} W.~Killing, Die Mechanik in den nichteuklidischen Raumformen,
{\it J. Reine Angew. Math.} {\bf 98} (1885), 1-48.

\bibitem{Lie1} H.~Liebmann, Die Kegelschnitte und die Planetenbewegung im 
nichteuklidischen Raum, {\it Berichte K\"onigl.~S\"achsischen Gesell. Wiss., Math.~Phys.~Klasse} {\bf 54} (1902), 393-423.

\bibitem{Lie2} H.~Liebmann, \"Uber die Zentralbewegung in der nichteuklidische
Geometrie, {\it Berichte K\"onigl.~S\"achsischen Gesell. Wiss., Math.~Phys.~Klasse} {\bf 55} (1903), 146-153.

\bibitem{Lip} R.~Lipschitz, Extension of the planet-problem to a space of $n$ dimensions and constant integral curvature, {\it Quart.~J.~Pure Appl.~Math.} {\bf 12} (1873), 349-370.

\bibitem{Mather} J.~Mather and R.~McGehee, Solutions of the collinear
four-body problem which become unbounded in finite time. In {\it Dynamical Systems
Theory and Applications}, pp.\ 573-587, ed.\ J.\ Moser, Lecture Notes in Physics,
Springer Verlag, New York, 1975.  

\bibitem{Mit} G.~Mittag-Leffler, Zur Biographie von Weierstrass, {\it Acta Math.}
{\bf 35} (1912), 29-65.

\bibitem{Pain} P.~Painlev\'e, {\it Le\c cons sur la th\'eorie analytique des \'equations
diff\'erentielles}, Hermann, Paris, 1897.

\bibitem{Sche} E.~Schering, Die Schwerkraft im Gaussischen Raume, 
{\it Nachr. K\"onigl. Gesell. Wiss. G\"ottingen} {\bf 15} (1870), 311-321.

\bibitem{Siegel} C.~L.~Siegel and J.~K.~Moser, {\it Lectures on Celestial
Mechanics}, Springer Verlag, Berlin-Heidelberg-New York, 1995.

\bibitem{Sun1} K.~F.~Sundman, Recherches sur le probl\`eme des trois
corps, {\it Acta.~Soc.~Sci.~Fennicae} {\bf 34}, 6 (1907), 1-43.

\bibitem{Sun2} K.~F.~Sundman, Nouvelle recherches sur le probl\`eme des trois
corps, {\it Acta. Soc.~Sci.~Fennicae} {\bf 35}, 9 (1909), 3-27.

\bibitem{Sun} K.~F.~Sundman, M\'emoire sur le probl\`eme des trois corps,
{\it Acta Math.} {\bf 36} (1912), 105-179.

\bibitem{Win} A.~Wintner, {\it The Analytical Foundations of Celestial
Mechanics}, Princeton University Press, Princeton, N.J., 1941.

\bibitem{Xia} Z.~Xia, The existence of noncollision singularities in
the $N$-body problem, {\it Annals Math.} {\bf 135} (1992), 411-468.



\end{thebibliography}
\end{document}